\documentclass[11pt,a4paper]{article}

\usepackage{amsmath}
\usepackage{amsfonts}
\usepackage{amssymb,color}

\usepackage{amsthm}
\usepackage{graphicx}
\usepackage[T1]{fontenc}     
\usepackage{color,graphicx}
\usepackage{epstopdf}
\usepackage{graphics}
\usepackage{psfrag}
\usepackage{url}
\usepackage{hyperref}
\usepackage{subcaption}
\theoremstyle{plain}

\newtheorem{theorem}{Theorem}

\newtheorem{defi}{Definition}
\newtheorem{lemma}{Lemma}
\newtheorem{notation}{Notation}
\newtheorem{prop}{Proposition}
\newtheorem{remark}{Remark}
\newtheorem{assumption}{Assumption}

\AtBeginDocument{

}

\title{On Strichartz estimates for a dispersion modulated by a time-dependent deterministic noise}

\author{Romain Duboscq \footnotemark[1]}

\begin{document}
\maketitle

\begin{center}
INSA de Toulouse \footnotemark[2] \\
IMT UMR CNRS 5219 \\
Universit\'e de Toulouse 
\end{center}

\footnotetext[1]{\texttt{romain.duboscq@insa-toulouse.fr}}
\footnotetext[2]{135 avenue de Rangueil 31077 Toulouse Cedex 4 France}

\abstract{\noindent 
We address the Cauchy problem for a nonlinear Schrödinger equation where the dispersion is modulated by a deterministic noise. The noise is understood as the derivative of a self-affine function of order $H\in (0,1)$. Due to the self-similarity of the noise, we obtain modified Strichartz estimates which enables us to prove the global well-posedness of the equation for $L^2$-supercritical nonlinearities. This is an occurence of regularization by noise in a purely deterministic context.}

\vspace{1em}
{\noindent \textbf{Key words:} Nonlinear Schrödinger equation; Strichartz estimate; Hardy-Littlewood-Sobolev inequality; Regularization by noise.
}

\vspace{1em}
\noindent
{\noindent \textbf{AMS 2010 subject classification:} Primary: 35Q55; Secondary: 60H15.}

\renewcommand{\thefootnote}{\fnsymbol{footnote}}

\allowdisplaybreaks

\section{Introduction}

In this paper, we wish to study the following type of nonlinear Schrödinger equation
\begin{equation}\label{eq:Main}\left\{\begin{array}{ll}
i\partial_t \psi(t,x) = \Delta \psi(t,x) \dot{X}_t + \lambda |\psi|^{2\sigma}\psi(t,x),\quad (t,x)\in [0,1]\times\mathbb{R}^d,
\\ \psi(0,\cdot) = \psi_0 \in L^2(\mathbb{R}^d),
\end{array}\right.
\end{equation}
where $\sigma \in \mathbb{R}^+$, $\lambda\in\mathbb{R}$,  $X$ is a deterministic continuous function.

This type of nonlinear Schrödinger equation modulated by a time-dependent function has been introduced in \cite{agrawal2001applications}, with $d=1$, to model the electric field of a light pulse travelling in an optical fiber with dispersion management. In a standard optical fiber, the electric field of a light pulse can be described as a soliton whose evolution is governed by a nonlinear Schrödinger equation (\textit{i.e.} Equation \eqref{eq:Main} with $\dot{X} = 1$). When propagating in the fiber, due to the dispersion, the soliton spreads and becomes difficult to detect since its amplitude decreases. This is a major issue when one wants to use optical fibers as communication devices. Since it is impossible to build fibers without dispersion, one way to avoid this problem is to engineer optical fibers with a dispersion varying rapidly around zero: these are called dispersion managed optical fibers.

By considering a random dispersion management, Marty \cite{marty2006splitting} derived a nonlinear Schrödinger equation with white noise dispersion, that is Equation \eqref{eq:Main} with $\dot{X} = \dot{W}$ where $W$ is a Wiener process. In \cite{de2010nonlinear}, de Bouard and Debussche proved that such equations are well-posed when the nonlinearity is $L^2$-subcritical, \textit{i.e.} $\sigma < 2/d$. Subsequently, Debussche and Tsutsumi improved this result to the $L^2$-critical case $\sigma = 2/d$ in \cite{debussche20111d}. Then, in \cite{belaouar2015numerical}, Belaouar, de Bouard and Debussche conducted numerical experiments  and conjectured that the critical nonlinear parameter should be $\sigma = 4/d$. In \cite{chouk2015nonlinear}, Chouk and Gubinelli studied a nonlinear Schrödinger equation modulated by a noise understood as the derivative of a $(\rho,\gamma)$-irregular function and solved the Cauchy problem in the $L^2$-critical case. Let us mention that, up to now, the only examples of $(\rho,\gamma)$-irregular functions are fractional Wiener processes (see \cite{CATELLIER20162323}). Finally, in \cite{duboscq2017stochastic}, the author and Réveillac showed that, in the context of the white noise dispersion, the equation is well-posed for $\sigma < 4/d$, that is for $L^2$-supercritical nonlinearities.

Since most of these result handle $L^2$-critical and supercritical nonlinearities, this indicates the strong stabilizing effect of dispersions modulated by a noise. This is reminiscent of the well-known regularization by noise effect (see \cite{flandoli2011random} for a survey) which is characterized by the improvement of the well-posedness of an evolution equation when introducing noise in it. This effect was originally discovered in the context of SDEs by Zvonkin in \cite{zvonkin1974transformation} where he was able, thanks to the Wiener process, to remove the singular drift from the equation and, thus, prove the Cauchy problem. This phenomenon was then generalized \cite{veretennikov1981strong,krylov2005strong,davie2007uniqueness} and also extended to the realm of SPDEs \cite{flandoli2010well,da2013strong,Gubinelli2013,gess2017long}. Let us remark that, to our knowledge, there is no explicit example of deterministic noise providing a regularization by noise effect.

Our main motivation to study Equation \eqref{eq:Main} is to prove that there can be a regularization effect by a deterministic noise for modulated nonlinear Schrödinger equations. Here, we investigate the case where $X$ is a self-affine function. This choice is motivated by the fractal property of these functions and the possibility of constructing explicit examples of them (see \cite{kamae1986characterization,bertoin1988mesure,okamoto2005remark}). Since $X$ is not differentiable, it is difficult in general to give a meaning to noise term $\dot{X}$ from Equation  \eqref{eq:Main}, whereas, in the Wiener setting, it is possible to employ Itô or Stratonovich's integration. Hence, as in  \cite{chouk2015nonlinear}, we rather consider the mild formulation of Equation \eqref{eq:Main}, which is given by
\begin{equation}\label{eq:MainMild}
\psi(t,x) = P_{0,t}\psi_0(x) - i\lambda\int_0^t P_{s,t}|\psi|^{2\sigma}\psi(s,x) ds,\quad\forall (t,x)\in[0,1]\times\mathbb{R}^d,
\end{equation}
with $(P_{s,t})_{0\leq s\leq t\leq 1}$ the propagator associated to the linear operator of \eqref{eq:Main}. That is, we have, $ \forall s,t\in [0,1]$, with $s\leq t$, $\forall \varphi\in\mathcal{C}^{\infty}_0(\mathbb{R}^d)$ and $\forall x\in\mathbb{R}^d$,
\begin{equation}\label{eq:propagator}
P_{s,t}\varphi(x) : = \mathcal{F}^{-1}\left( e^{i|\xi|^2(X_t-X_s)}\hat{\varphi}(\xi) \right)(x) = \frac{1}{(4\pi(X_t-X_s))^{d/2}}\int_{\mathbb{R}^d} e^{i\frac{|x-y|^2}{4(X_t-X_s)}}\varphi(y)dy,
\end{equation}
where we denote $\mathcal{F}^{-1}$ the inverse Fourier transform and $\hat{\varphi}$ the Fourier transform of $\varphi$.

In order to solve the Cauchy problem of Equation \eqref{eq:MainMild}, we investigate Strichartz estimates to apply a fixed-point argument. This is the classical strategy for this type of nonlinear dispersive equations \cite{cazenave2003semilinear}. Our argument somehow follow the one from \cite{duboscq2017stochastic} in the sense that we start by deriving a modified Hardy-Littlewood-Sobolev inequality adapted to our situation. The regularization effect will take its roots in this new inequality and mainly relies on the scaling invariance of self-affine functions. From there, the Strichartz estimates are directly obtained by the usual $TT^*$ method \cite{keel1998endpoint}.

The rest of the paper is organized as follows: in Section \ref{sec:SelfMain} we introduce the class of self-affine functions that we consider and describe our main results, in Section \ref{sec:Strichartz} we derive the Strichartz estimates associated to modulated dispersion and finally, in Section \ref{sec:Cauchy}, we solve the Cauchy problem associated to Equation \eqref{eq:MainMild}.

\section{Self-affine functions and main results}\label{sec:SelfMain}

Let us start by recalling the definition of self-affine functions. Here, we follow the definition given by Kamae in \cite{kamae1986characterization}. 

\begin{defi}
Let $(X_t)_{0\leq t\leq 1}$ be a continuous real-valued function such that $X_0 = 0$ and $X \neq 0$, and $b\in\mathbb{N}$ such that $b\geq 2$. We say that $X$ is a self-affine function of order $H\in (0,1]$ and base $b$ if there exists a finite set of real-valued functions $\mathcal{Y}^{X} = \left\{(Y_t^{(m)})_{0\leq t\leq 1}\right\}_{1\leq m\leq N}$, $N\in\mathbb{N}$, such that
\begin{enumerate}
\item for any $n\in\mathbb{N}$ and $j\in\{0,\cdots,b^n -1\}$, there exists $Y\in\mathcal{Y}^{X}$ such that
\begin{equation}\label{eq:selfaffine}
X_{\frac{j+t}{b^n}} - X_{\frac{j}{b^n}} = b^{-Hn} Y_t,\quad \forall t\in[0,1],
\end{equation}
\item for any $m\in\mathbb{N}$, $k\in\{0,\cdots,b^m-1\}$ and $Y\in\mathcal{Y}^{X}$, there exists integers $n$ and $i$ with $m\leq n$ and 
\begin{equation*}
kb^{n-m}\leq i<(k+1)b^{n-m},
\end{equation*}
such that \eqref{eq:selfaffine} holds.
\end{enumerate}
\end{defi}

Throughout this paper, we will consider a specific subset of self-affine functions that satisfy the following assumption.
\begin{assumption}\label{asm:sa}
Let $X$ be a self-affine function. There exists a constant $C>0$ such that
\begin{equation}\label{eq:selfaffineAsm}
\min_{n\in\mathbb{N},j\in\{0,\cdots,b^n-1\}} b^{Hn}|X_{\frac{j+t}{b^n}} - X_{\frac{j}{b^n}} | = C^{-1}.
\end{equation}
\end{assumption}

\begin{remark}This assumption is required to prove Theorem \ref{thm:mHLS} below since it prevent singularities in the study of the discretized inequality.
\end{remark}
\begin{notation} We denote $\mathcal{A}$ the set of self-affine functions that satisfy Assumption \ref{asm:sa}.
\end{notation}

 In order to prove that $\mathcal{A}$ is not empty, we provide below a class of functions that belongs in $\mathcal{A}$ and that were introduced in \cite{bertoin1988mesure}. Let $a,b\in\mathbb{N}$ such that $2\leq a<b$ and $M,D\in\mathcal{C}([0,1])$ two functions from $[0,1]$ to itself such that
\begin{enumerate}
\item $M(0) = D(1) = 0$ and $M(1) = D(0) = 1$,
\item $M$ and $D$ are affine functions on each interval $[jb^{-1},(j+1)b^{-1}]$ with $j\in\{0,\cdots,a-1\}$,
\item $|M((j+1)b^{-1}) - M(jb^{-1})| = |D((j+1)b^{-1}) - D(jb^{-1})| = a^{-1}$.
\end{enumerate}
We now consider the sequence $(X^{(n)})_{n\geq 0}\subset\mathcal{C}([0,1])$ of functions constructed by following the procedure
\begin{enumerate}
\item $X^{(1)} = M$,
\item for any $n\geq 1$, for any $j\in\{0,\cdots,b^n-1\}$ and any $t\in [jb^{-n},(j+1)b^{-n}]$, we define 
\begin{equation*}
X^{(n+1)}_{(j+t)b^{-n}} : = X^{(n)}_{jb^{-n}} + a^{-n}M(t),
\end{equation*}
if $X^{(n)}$ is increasing on $[jb^{-n},(j+1)b^{-n}]$ and
\begin{equation*}
X^{(n+1)}_{(j+t)b^{-n}} : = X^{(n)}_{jb^{-n}} - a^{-n}(1 - D(t)),
\end{equation*}
if $X^{(n)}$ is decreasing on $[jb^{-n},(j+1)b^{-n}]$.
\end{enumerate}
We have the following result concerning the limiting function constructed this way.
\begin{theorem}(see \cite{bertoin1988mesure})
The sequence $(X^{(n)})_{n\geq 0}$ converges in $\mathcal{C}([0,1])$ to a function $X$ which is self-affine of order $H = \log(b)/\log(a)$ and base $b$.
\end{theorem}

We can see that, for any $n\in\mathbb{N}^*$, $X^{(n)}$ is a linear interpolation of $X$ and that the assumption \eqref{eq:selfaffineAsm} is satisfied since, for any $j\in\{0,\cdots,b^n -1\}$,
\begin{equation*}
a^{n}|X_{(j+1)b^{-n}} - X_{jb^{-n}}| = 1.
\end{equation*}

\begin{remark}
In order to illustrate this construction, we present in Figure \ref{fig:selfaffine} two examples of such functions.
\end{remark}
\begin{figure}
\centering
\begin{subfigure}[b]{0.45\textwidth}
\includegraphics[scale=0.4]{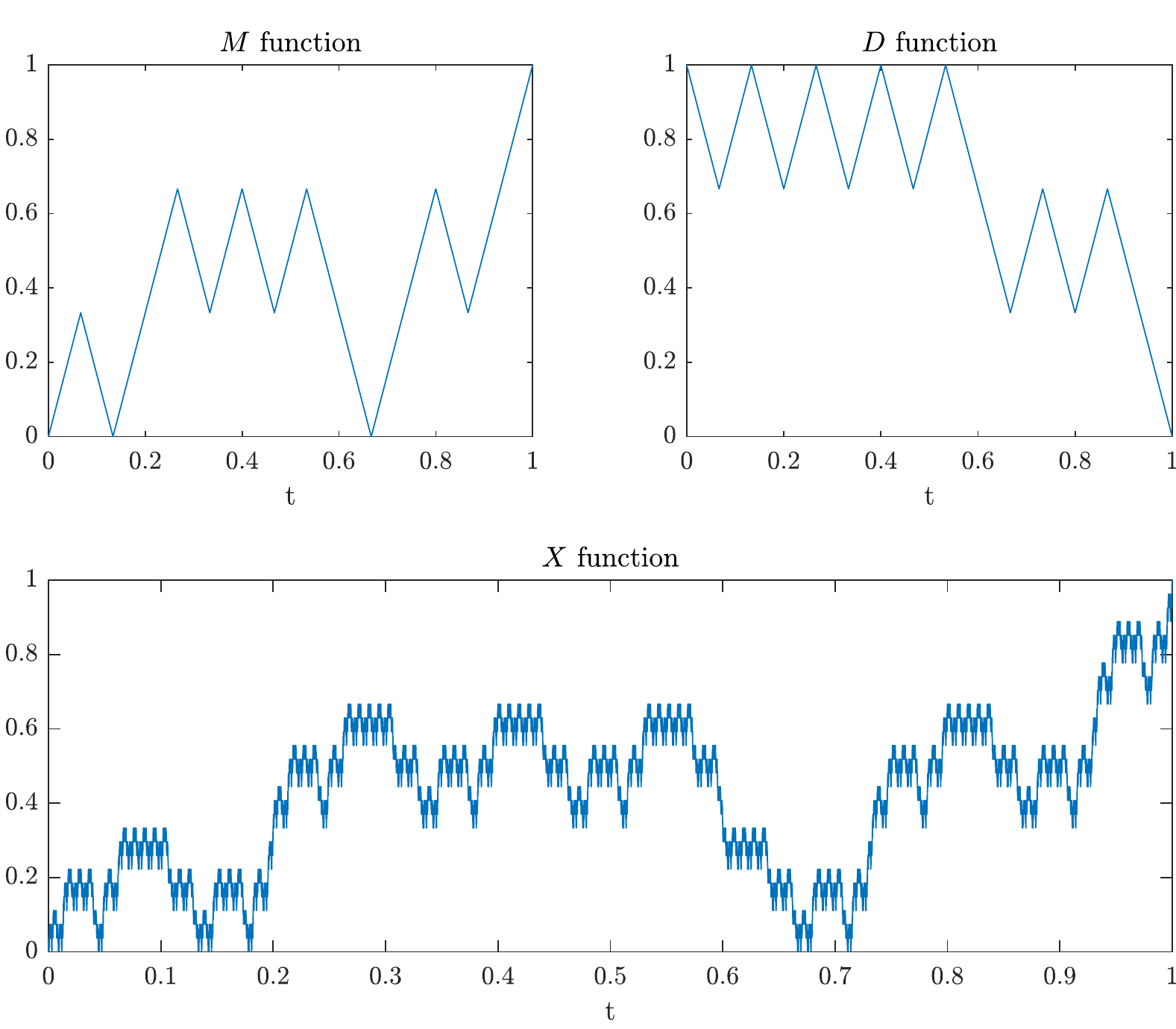}
\caption{$a = 15, b = 3$ and $H \simeq 0.4 $.}
\end{subfigure}
\begin{subfigure}[b]{0.45\textwidth}
\includegraphics[scale=0.4]{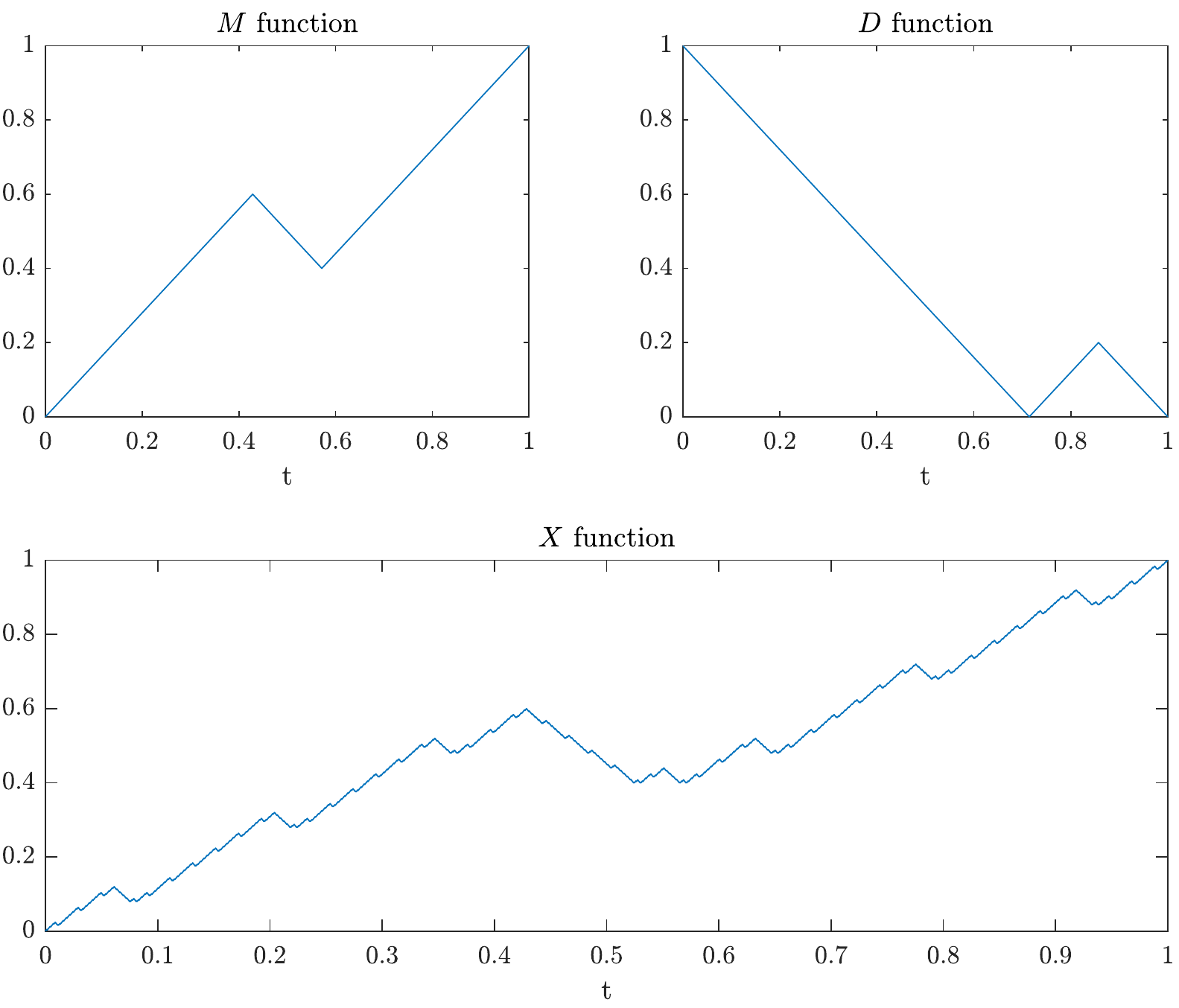}
\caption{$a = 7, b = 5$ and $H \simeq 0.8$.}
\end{subfigure}
\caption{Two constructions of a self-affine function following the procedure in \cite{bertoin1988mesure}.\label{fig:selfaffine}}
\end{figure}

Let us now introduce the following definition which is a modification of the standard admissible pairs for Strichartz estimates.
\begin{defi}
For any $(q,p)\in (2,\infty)^2$ and $H\in(0,1)$, we say that $(q,p)$ is $H$-admissible if
\begin{equation*}
\frac{2}{q} = dH\left(\frac 1 2 - \frac 1 p \right).
\end{equation*}
\end{defi}
We can now state our first main result on the dispersive properties of the propagator $P$ defined in \eqref{eq:propagator}.
\begin{theorem}\label{thm:Strichartz}
Let $T\in (0,1]$, $X\in\mathcal{A}$ of order $H\in(0,1)$ and $(q,p)$ $H$-admissible. Then, there exists two constants $C_1,C_2>0$ which depends on $d$, $p$ and $q$ such that, $\forall f\in L^2(\mathbb{R}^d)$ and $\forall g\in L^{r'}([0,T];L^{l'}(\mathbb{R}^d))$, the following inequalities holds 
\begin{align}
\left\|P_{0,\cdot} f \right\|_{L^q([0,T];L^p(\mathbb{R}^d))}&\leq  C_1\|f\|_{L^2}\label{eq:Stri1},
\\ \left\|\int_0^T P_{s,\cdot} g(s)ds \right\|_{L^q([0,T];L^p(\mathbb{R}^d))} &\leq C_2 \|g\|_{L^{r'}([0,T];L^{l'}(\mathbb{R}^d))},\label{eq:Stri2}
\end{align}
for any $(r,l)$ $H$-admissible.
\end{theorem}

Thanks to the dispersive estimates of Theorem \ref{thm:Strichartz} and by a standard fixed point argument, we can solve the Cauchy problem for Equation \eqref{eq:MainMild}. Thus, we have the next theorem.
\begin{theorem}\label{thm:Cauchy}
Let $X\in\mathcal{A}$ of order $H\in(0,1)$, $\psi_0\in L^2(\mathbb{R}^d)$ and $\sigma < \frac{2}{dH}$. There exists a unique solution $\psi \in L^{r}([0,1];L^{2\sigma + 2}(\mathbb{R}^d))$ to Equation \eqref{eq:MainMild} where $r$ is such that $(r,2\sigma + 2)$ is $H$-admissible.
\end{theorem}

We can see that the order of the self-affine function $X$ directly affects the bound on the exponent of the nonlinearity. More precisely, since the order of $H$ is representative of the regularity of $X$, the more $X$ is irregular the bigger is the critical exponent of the nonlinearity. 

\begin{remark} We remark that this type of result is similar, in a sense, to \cite[Theorem 1.9]{CATELLIER20162323}, where Catellier and Gubinelli prove that, for a SDE driven by a fractional Brownian motion of Hurst parameter $H$, there exists a solution if the drift belongs to $\mathcal{C}^{\alpha}$ with $\alpha>1 - 1/2H$. Thus, a rougher fractional Brownian motion gives a stronger regularization effect.
\end{remark}

We finally note that it would be interesting to investigate ODE driven by self-affine functions and, in a larger sense, to look out for an explicit examples of regularization by noise in ODE.

\section{Strichartz estimates for the modulated dispersion}\label{sec:Strichartz}

\subsection{A Hardy-Littlewood-Sobolev inequality}
Our first step toward the proof of Theorem \ref{thm:Strichartz} is to deduce the following modified Hardy-Littlewood-Sobolev inequality (see \cite{hardy1928some,hardy1932some,sobolev1938theorem} for the classical Hardy-Littlewood-Sobolev inequality).

\begin{theorem}\label{thm:mHLS}
Let $X\in\mathcal{A}$ of order $H\in (0,1]$, $\alpha\in (0,1)$ and $f\in L^p([0,1])$ and $g\in L^q([0,1])$ such that $p,q\in(1,\infty)$ and
\begin{equation*}
2-\alpha H = \frac 1 p + \frac 1 q.
\end{equation*}
Then, there exists a constant $C>0$ which depends on $p$ and $q$ such that the following inequality holds
\begin{equation}\label{eq:cHLS}
\left|\int_0^1\int_0^1 f(t)|X_t-X_s|^{-\alpha} g(s) ds dt \right| \leq C \|f\|_{L^p([0,T])} \|g\|_{L^q([0,T])}.
\end{equation}

\end{theorem}

\begin{proof}

First, without loss of generality, we can assume that $f$ and $g$ are positive functions and, by a density argument, that they are continuous. For any $n\geq 1$ and $j,k$ such that $0\leq j,k\leq b^n$, we consider two uniform discretizations of $[0,1]$ given by
\begin{equation*}
t_{j}^{(n)} : = \frac{j}{b^n}\quad\textrm{and}\quad s_{k}^{(n)} : = \frac{k}{b^n}.
\end{equation*}
Furthermore, we consider the following approximations at the points $\left(t_k^{(n)}\right)_{0\leq k\leq b^{n}}$
\begin{enumerate}
\item $f$ and $g$ by the step functions $f^{(n)}$ and $g^{(n)}$,
\item $X$ by a linear interpolation $X^{(n)}$.
\end{enumerate}
We now introduce the following approximation of the integral on the left-hand-side of \eqref{eq:cHLS}, that is
\begin{equation}
I^{(\varepsilon,n)} = \sum_{j = 0}^{b^n-1}\sum_{k = 0}^{b^n-1} f\left(t^{(n)}_j\right) g\left(s^{(n)}_k\right) \int_{t^{(n)}_j}^{t^{(n)}_{j+1}}\int_{s^{(n)}_k}^{s^{(n)}_{k+1}}\left(|X_t^{(n)}-X_s^{(n)}|+\varepsilon\right)^{-\alpha}  ds dt.
\end{equation}
In order to obtain \eqref{eq:cHLS} from this integral, we need to prove that, up to a constant, it is bounded by the $L^p$-norm of $f^{(n)}$ and the $L^q$-norm of $g^{(n)}$. Then, we use Fatou's Lemma to let $n\to\infty$ and then the monotone convergence theorem to let $\varepsilon\to 0$.

To deduce the desired bound on $I^{(\varepsilon,n)}$, we need to estimate, for any $\ell,k\in\{0,\cdots b^n\}$, the integral
\begin{align*}
\iota_{k,\ell}^{(\varepsilon,n)} : = \int_{t_{k}^{(n)}}^{t_{k+1}^{(n)}}\int_{t_{\ell}^{(n)}}^{t_{\ell+1}^{(n)}}\left(|X_t^{(n)} - X_s^{(n)}| + \varepsilon\right)^{-\alpha} ds dt.
\end{align*}
We directly obtain that
\begin{equation*}
\iota_{k,\ell}^{(\varepsilon,n)} \leq \iota_{k,\ell}^{(0,n)} : = \iota_{k,\ell}^{(n)},
\end{equation*}
and, moreover, we have the following result.
\begin{lemma}\label{lem:estiota} Let $n\in\mathbb{N}^*$ and $k,\ell\in\{0,\cdots b^n\}$. We have
\begin{align*}
\iota_{k,\ell}^{(n)} = \frac1{(1-\alpha)(2-\alpha)}\frac{t_{\ell+1}^{(n)} - t_{\ell}^{(n)}}{X_{t^{(n)}_{\ell+1}} - X_{t^{(n)}_{\ell}}}\frac{t_{k+1}^{(n)} - t_{k}^{(n)}}{X_{t^{(n)}_{k+1}} - X_{t^{(n)}_{k}}}\varDelta_{\ell,k}^{(n)} |X_t - X_s|^{2-\alpha},
\end{align*}
where
\begin{align*}
 \varDelta_{\ell,k}^{(n)} |X_t - X_s|^{2-\alpha} : =&
 \\ &\hspace{-5em}  |X_{t^{(n)}_{k+1}} - X_{t^{(n)}_{\ell}}|^{2-\alpha} + |X_{t^{(n)}_{k}} - X_{t^{(n)}_{\ell+1}}|^{2-\alpha} - |X_{t^{(n)}_{k+1}} - X_{t^{(n)}_{\ell+1}}|^{2-\alpha} - |X_{t^{(n)}_{k}} - X_{t^{(n)}_{\ell}}|^{2-\alpha} .
\end{align*}
\end{lemma}
The proof of Lemma \ref{lem:estiota} is postponed in Section \ref{sec:estiota}.
Since $X$ is a self-affine function and thanks to \eqref{eq:selfaffine} and \eqref{eq:selfaffineAsm}, we have that there exists $Z^{(1)},Z^{(2)}\in\mathcal{Y}$ such that
\begin{align}
|X_{t^{(n)}_{k+1}} - X_{t^{(n)}_{k}}| &= b^{-Hn} |Z^{(1)}_1| \geq b^{-Hn} C^{-1},\label{ineq:integral0}\\
|X_{t^{(n)}_{\ell+1}} - X_{t^{(n)}_{\ell}}| &= b^{-Hn} |Z^{(2)}_1| \geq b^{-Hn} C^{-1},\label{ineq:integral1}
\end{align}
and, moreover, it follows from \eqref{eq:selfaffine}  that there exists $\eta \in\mathbb{N}^{N}$ with $|\eta| = k-\ell$ such that
\begin{equation*}
 X_{t^{(n)}_{k}} - X_{t^{(n)}_{\ell}} = b^{-Hn} \sum_{j = 1}^N \eta_j Y^{(j)}_1 = : b^{-Hn} \delta_{k,\ell} X.
\end{equation*}
Now, assume, for instance, that $X_{t^{(n)}_{k+1}} \geq X_{t^{(n)}_{\ell+1}} \geq X_{t^{(n)}_{\ell}} \geq X_{t^{(n)}_{k}}$ (the other cases follow from similar computations), we obtain, thanks to Taylor-Lagrange's formula,
\begin{align}
 \varDelta_{\ell,k}^{(n)} |X_t - X_s|^{2-\alpha} &\nonumber
 \\ &\hspace{-7em}=  (X_{t^{(n)}_{k+1}} - X_{t^{(n)}_{\ell+1}})^{2-\alpha} + (X_{t^{(n)}_{k}} - X_{t^{(n)}_{\ell}})^{2-\alpha} - (X_{t^{(n)}_{k+1}} - X_{t^{(n)}_{\ell}})^{2-\alpha} -(X_{t^{(n)}_{k}} - X_{t^{(n)}_{\ell+1}})^{2-\alpha}\nonumber
 \\ &\hspace{-7em}= b^{-Hn(2-\alpha)} \left( \left(\delta_{k,\ell} X + Z^{(1)}_1 - Z^{(2)}_1\right)^{2-\alpha} + \left(\delta_{k,\ell} X\right)^{2-\alpha} - \left(\delta_{k,\ell} X + Z^{(1)}_1 \right)^{2-\alpha} - \left(\delta_{k,\ell} X - Z^{(2)}_1 \right)^{2-\alpha} \right)\nonumber
 \\&\hspace{-7em}\lesssim b^{-Hn(2-\alpha)} |Z_1^{(1)}| \lesssim b^{-Hn(2-\alpha)}.\label{ineq:integral2}
\end{align}
Hence, we deduce from Lemma \ref{lem:estiota}, \eqref{ineq:integral0}, \eqref{ineq:integral1} and \eqref{ineq:integral2} that
\begin{align*}
\iota_{k,\ell}^{(n)} &\lesssim b^{-Hn(2-\alpha)}\frac{t_{\ell+1}^{(n)} - t_{\ell}^{(n)}}{|X_{t^{(n)}_{\ell+1}} - X_{t^{(n)}_{\ell}}|}\frac{t_{k+1}^{(n)} - t_{k}^{(n)}}{|X_{t^{(n)}_{k+1}} - X_{t^{(n)}_{k}}|}
\\ &\lesssim C^2 b^{Hn\alpha - 2n}.
\end{align*}
We can now proceed to estimate $I^{(\varepsilon,n)}$. Since $H\alpha - 2 = -1/p-1/q$ and thanks to Jensen's inequality, we have
\begin{align*}
I^{(\varepsilon,n)} &\lesssim C^2 \sum_{j = 0}^{b^n-1}\sum_{k = 0}^{b^n-1} f\left(t^{(n)}_j\right) g\left(s^{(n)}_k\right) b^{-n(1/p+1/q)}
\\ &\lesssim C^2 \left(\sum_{j = 0}^{b^n-1} f\left(t^{(n)}_j\right)^p b^{-n}\right)^{1/p} \left(\sum_{k = 0}^{b^n-1} g\left(s^{(n)}_k\right)^q b^{-n}\right)^{1/q} = C^2 \|f^{(n)}\|_{L^p([0,1])}\|g^{(n)}\|_{L^q([0,1])},
\end{align*}
which concludes the proof.
\end{proof}

\subsection{Proof of Theorem \ref{thm:Strichartz}}

Thanks to the previous result and by following the $TT^*$ method \cite{keel1998endpoint,cazenave2003semilinear}, we can now prove Theorem \ref{thm:Strichartz}. We easily deduce the following preliminary result thanks to the Fourier formulation of $(P_{s,t})_{0\leq s\leq t\leq 1}$ given in \eqref{eq:propagator}.

\begin{lemma}
Let $0\leq s\leq t\leq 1$. We have, $\forall \varphi\in L^2(\mathbb{R}^d)$,
\begin{equation}\label{eq:isometry}
\|P_{s,t}\varphi\|_{L^2(\mathbb{R}^d)} = \|\varphi\|_{L^2(\mathbb{R}^d)},
\end{equation}
Moreover the adjoint of $P$, denoted $P^*$, is such that, $\forall \varphi\in L^2(\mathbb{R}^d)$,
\begin{equation*}
P^*_{s,t}\varphi =  \mathcal{F}^{-1}\left( e^{-i|\xi|^2(X_t-X_s)}\hat{\varphi}(\xi) \right) = P_{t,s}\varphi,
\end{equation*}
and,
\begin{equation*}
P_{0,s}^*P_{0,t} = P_{s,t}.
\end{equation*}
\end{lemma}
It follows from the formulation \eqref{eq:propagator} of $P$, in the space variables, that, $\forall \varphi\in L^1(\mathbb{R}^d)$,
\begin{equation*}
\|P_{s,t}\varphi\|_{L^{\infty}(\mathbb{R}^d)} \lesssim \frac{1}{|X_t-X_s|^{d/2}}  \|\varphi\|_{L^1(\mathbb{R}^d)}.
\end{equation*}
Hence, by Riesz-Thorin's theorem, we deduce that, $\forall p\in [2,\infty]$, $\forall \varphi\in L^{p'}(\mathbb{R}^d)$,
\begin{equation}\label{ineq:dispest}
\|P_{s,t}\varphi\|_{L^p(\mathbb{R}^d)} \lesssim \frac{1}{|X_t-X_s|^{d(1/2-1/p)}} \|\varphi\|_{L^{p'}(\mathbb{R}^d)}.
\end{equation}
where $p'$ is the Hölder conjugate of $p$. 

Let $T\in (0,1]$ and $(q,p)$ $H$-admissible. We now consider the integral, $\forall f,g\in \mathcal{C}([0,T], \mathcal{C}^{\infty}_0(\mathbb{R}^d))$,
\begin{equation*}
I(f,g) = \left| \int_0^T\int_0^T \langle P_{0,t} f(s), P_{0,s} g(t)\rangle_{L^2} ds dt \right| = \left| \int_0^T\int_0^T \langle P_{s,t} f(s), g(t)\rangle_{L^2} ds dt \right|
\end{equation*}
Thanks to Hölder's inequality, \eqref{ineq:dispest} and Theorem \ref{thm:mHLS}, we obtain the following inequality, $\forall p\in(1,\infty)$,
\begin{align*}
I(f,g) &\leq  \int_0^T\int_0^T \|P_{s,t}f(s)\|_{L^{p}(\mathbb{R}^d)}\|g(t)\|_{L^{p'}(\mathbb{R}^d)} ds dt
\\&\lesssim\int_0^T\int_0^T |X_t-X_s|^{d(1/2-1/p)}\|f(t)\|_{L^{p'}(\mathbb{R}^d)}\|g(s)\|_{L^{p'}(\mathbb{R}^d)} ds dt
\\&\lesssim \|f\|_{L^{q_1}([0,T],L^{p'}(\mathbb{R}^d))}\|g\|_{L^{q_2}([0,T],L^{p'}(\mathbb{R}^d))},
\end{align*}
where $q_1,q_2\in (1,\infty)$ are such that
\begin{equation*}
2 - dH\left(\frac12-\frac1p\right) = \frac 1{q_1} + \frac1{q_2}.
\end{equation*}
By taking $q_1 = q_2 = q'$, the previous inequality becomes
\begin{equation*}
\frac2q = d H\left(\frac12-\frac1p\right).
\end{equation*}
Thus, we obtain that
\begin{equation}\label{eq:StriHomoDual}
\left\|\int_0^T P_{0,s}^* f(s) ds \right\|_{L^2(\mathbb{R}^d)}^2 = I(f,f) \lesssim  \|f\|_{L^{q'}([0,T],L^{p'}(\mathbb{R}^d))}^2,
\end{equation}
and, by a duality argument, we deduce
\begin{equation}\label{eq:Stri2b}
\left\| \int_0^T P_{s,\cdot}f(s) ds \right\|_{L^{q}([0,T],L^p(\mathbb{R}^d))} \lesssim\|f\|_{L^{q'}([0,T],L^{p'}(\mathbb{R}^d))}.
\end{equation}
By duality, we have that
\begin{equation*}
\|P_{0,t}f\|_{L^q([0,T];L^p(\mathbb{R}^d))} = \sup_{ \|g\|_{L^{q'}([0,T];L^{p'}(\mathbb{R}^d))} = 1} \left| \int_{0}^T \langle P_{0,t}f,g(t) \rangle_{L^2}dt\right|,
\end{equation*}
and, furthermore, thanks to \eqref{eq:StriHomoDual}, $\forall f\in L^2(\mathbb{R}^d)$ and $\forall g\in L^{q'}([0,T];L^{p'}(\mathbb{R}^d))$,
\begin{align*}
\left|\int_{0}^T \langle P_{0,t}f,g(t) \rangle_{L^2} dt \right|&= \left|\left\langle f, \int_0^T P_{0,t}^*g(t)\right\rangle_{L^2}\right| \leq \|f\|_{L^2(\mathbb{R}^d)} \left\|\int_0^T P_{0,t}^*g(t) ds \right\|_{L^2(\mathbb{R}^d)}
\\ &\lesssim \|f\|_{L^2(\mathbb{R}^d)} \|g\|_{L^{q'}([0,T],L^{p'}(\mathbb{R}^d))},
\end{align*}
which gives \eqref{eq:Stri1}. In order to obtain \eqref{eq:Stri2}, we remark that, by \eqref{eq:StriHomoDual},
\begin{align*}
\left\|\int_0^T P_{s,\cdot} f(s) ds\right\|_{L^q([0,T];L^p(\mathbb{R}^d))} &\leq \int_0^T \left\|P_{s,\cdot} f(s) \right\|_{L^q([0,T];L^p(\mathbb{R}^d))} ds
\\ &\leq  \int_0^T \|f(s)\|_{L^2(\mathbb{R}^d))} ds = \|f\|_{L^1([0,T];L^2(\mathbb{R}^d))}.
\end{align*}
Inequality \eqref{eq:Stri2} follows from an interpolation argument between the previous inequality and \eqref{eq:Stri2b}.

\subsection{Proof of Lemma \ref{lem:estiota}}\label{sec:estiota}

We remark that we have
\begin{equation*}
X^{(n)}_t = X_{t^{(n)}_k} + (t -t^{(n)}_k)(X_{t^{(n)}_{k+1}} - X_{t^{(n)}_k})(t_{k+1}^{(n)} - t_{k}^{(n)})^{-1},
\end{equation*}
and, for any $s\in[t^{(n)}_{\ell},t^{(n)}_{\ell+1}]$,
\begin{equation*}
X^{(n)}_s = X_{t^{(n)}_{\ell}} + (s -t^{(n)}_{\ell})(X_{t^{(n)}_{\ell+1}} - X_{t^{(n)}_{\ell}})(t_{\ell+1}^{(n)} - t_{\ell}^{(n)})^{-1}.
\end{equation*}
We also remark that, thanks to assumption \eqref{eq:selfaffineAsm}, we have 
\begin{equation*}
X_{t^{(n)}_{m+1}} - X_{t^{(n)}_{m}} \neq 0,
\end{equation*}
for any $m$ in $\{0,\cdots,b^n\}$. 
We now decompose our proof in 3 steps which depend on the following assumptions
\begin{enumerate}
\item[1] there exists a (unique) $t^*\in]t^{(n)}_{\ell},t^{(n)}_{\ell+1}[$ such that
\begin{equation}\label{eq:tstarassm}
X_t^{(n)} - X_{t^*}^{(n)} = 0.
\end{equation}
\item[2] we have $\forall s\in [t^{(n)}_{\ell},t^*]$,
\begin{equation}\label{eq:asmSignLem}
X_t^{(n)} - X_{s}^{(n)} \geq 0.
\end{equation}
\end{enumerate}
\noindent\textbf{Step 1: Assumptions 1 and 2 are verified} 

We have, $\forall t\in[t^{(n)}_k,t^{(n)}_{k+1}]$,
\begin{align*}
\int_{t_{\ell}^{(n)}}^{t_{\ell+1}^{(n)}}|X_t^{(n)} - X_s^{(n)}|^{-\alpha} ds &= \int_{t_{\ell}^{(n)}}^{t^*}|X_t^{(n)} - X_s^{(n)}|^{-\alpha}ds + \int_{t^*}^{t_{\ell+1}^{(n)}}|X_t^{(n)} - X_s^{(n)}|^{-\alpha}ds
\\ &\hspace{-12em}=\int_{t_{\ell}^{(n)}}^{t^*}(X_t^{(n)} - X_s^{(n)})^{-\alpha}ds + \int_{t^*}^{t_{\ell+1}^{(n)}}(X_s^{(n)} - X_t^{(n)})^{-\alpha}ds
\\ &\hspace{-12em}=\frac{1}{(1-\alpha)}\frac{t_{\ell+1}^{(n)} - t_{\ell}^{(n)}}{X_{t^{(n)}_{\ell+1}} - X_{t^{(n)}_{\ell}}}\left((X_t^{(n)} - X_{t^{(n)}_{\ell}})^{1-\alpha} + (X_{t^{(n)}_{\ell+1}} - X_t^{(n)})^{1-\alpha}\right).
\end{align*}
Hence, by denoting, for any $m$ in $\{0,\cdots,b^n\}$,
\begin{equation*}
\zeta_{n,m} : = \frac{t_{m+1}^{(n)} - t_{m}^{(n)}}{X_{t^{(n)}_{m+1}} - X_{t^{(n)}_{m}}},
\end{equation*}
we obtain
\begin{align}
\iota_{k,\ell}^{(n)}=& \; \frac{\zeta_{n,\ell}}{1-\alpha} \int_{t_{k}^{(n)}}^{t_{k+1}^{(n)}}\left((X_t^{(n)} - X_{t^{(n)}_{\ell}})^{1-\alpha} + (X_{t^{(n)}_{\ell+1}} - X_t^{(n)})^{1-\alpha}\right)dt\nonumber
\\ =& \;\frac{\zeta_{n,\ell}}{1-\alpha}\frac{\zeta_{n,k}}{2-\alpha}\left((X_{t^{(n)}_{k+1}} - X_{t^{(n)}_{\ell}})^{2-\alpha} - (X_{t^{(n)}_{k}} - X_{t^{(n)}_{\ell}})^{2-\alpha} \right)\nonumber
\\ &+ \;\frac{\zeta_{n,\ell}}{1-\alpha}\frac{\zeta_{n,k}}{2-\alpha}\left(-(X_{t^{(n)}_{\ell+1}} - X_{t^{(n)}_{k+1}})^{2-\alpha} + (X_{t^{(n)}_{\ell + 1}} - X_{t^{(n)}_{k}})^{2-\alpha} \right).\label{eq:LmRes1}
\end{align}
\noindent\textbf{Step 2: Assumption 1 is verified and 2 is not}

If we  assume the opposite inequality in \eqref{eq:asmSignLem}, we obtain that
\begin{align*}
\int_{t_{\ell}^{(n)}}^{t_{\ell+1}^{(n)}}|X_t^{(n)} - X_s^{(n)}|^{-\alpha} ds
\\ &\hspace{-12em}=\frac{1}{(1-\alpha)}\frac{t_{\ell+1}^{(n)} - t_{\ell}^{(n)}}{X_{t^{(n)}_{\ell+1}} - X_{t^{(n)}_{\ell}}}\left(-(X_{t^{(n)}_{\ell}}- X_t^{(n)})^{1-\alpha} - ( X_t^{(n)} - X_{t^{(n)}_{\ell+1}})^{1-\alpha}\right),
\end{align*}
which leads to
\begin{align}
\iota_{k,\ell}^{(n)}=& \; - \frac{\zeta_{n,\ell}}{1-\alpha} \int_{t_{k}^{(n)}}^{t_{k+1}^{(n)}}\left((X_{t^{(n)}_{\ell}}-X_t^{(n)})^{1-\alpha} + (X_t^{(n)} - X_{t^{(n)}_{\ell+1}})^{1-\alpha}\right)dt\nonumber
\\ =& \;\frac{\zeta_{n,\ell}}{1-\alpha}\frac{\zeta_{n,k}}{2-\alpha}\left((X_{t^{(n)}_{\ell}}- X_{t^{(n)}_{k+1}})^{2-\alpha} - (X_{t^{(n)}_{k}} - X_{t^{(n)}_{\ell}})^{2-\alpha} \right)\nonumber
\\ &+ \;\frac{\zeta_{n,\ell}}{1-\alpha}\frac{\zeta_{n,k}}{2-\alpha}\left(-(X_{t^{(n)}_{k+1}} - X_{t^{(n)}_{\ell+1}})^{2-\alpha} + (X_{t^{(n)}_{k}} - X_{t^{(n)}_{\ell + 1}})^{2-\alpha} \right).\label{eq:LmRes2}
\end{align}
\noindent\textbf{Step 3: Assumption 1 is not verified}

Finally, if we assume that there does not exists a $t^*\in]t^{(n)}_{\ell},t^{(n)}_{\ell+1}[$ such that \eqref{eq:tstarassm} holds, we compute directly, if $X^{(n)}_s \geq X^{(n)}_t$ for all $(t,s) \in [t^{(n)}_k,t^{(n)}_{k+1}] \times[t^{(n)}_{\ell},t^{(n)}_{\ell+1}]$,
\begin{align}
\iota_{k,\ell}^{(n)}=& \;\frac{\zeta_{n,\ell}}{1-\alpha}\frac{\zeta_{n,k}}{2-\alpha}\left(-(X_{t^{(n)}_{\ell+1}} - X_{t^{(n)}_{k+1}})^{2-\alpha} + (X_{t^{(n)}_{\ell + 1}} - X_{t^{(n)}_{k}})^{2-\alpha} \right)\nonumber
\\ &+ \;\frac{\zeta_{n,\ell}}{1-\alpha}\frac{\zeta_{n,k}}{2-\alpha}\left((X_{t^{(n)}_{\ell}} - X_{t^{(n)}_{k+1}})^{2-\alpha} - (X_{t^{(n)}_{\ell}} - X_{t^{(n)}_{k}})^{2-\alpha} \right).\label{eq:LmRes3}
\end{align}
and, if $X^{(n)}_s \leq X^{(n)}_t$ for all $(t,s) \in [t^{(n)}_k,t^{(n)}_{k+1}] \times[t^{(n)}_{\ell},t^{(n)}_{\ell+1}]$,
\begin{align}
\iota_{k,\ell}^{(n)}=& \;\frac{\zeta_{n,\ell}}{1-\alpha}\frac{\zeta_{n,k}}{2-\alpha}\left(-(X_{t^{(n)}_{k+1}} - X_{t^{(n)}_{\ell+1}})^{2-\alpha} + (X_{t^{(n)}_{k}} - X_{t^{(n)}_{\ell + 1}})^{2-\alpha} \right)\nonumber
\\ &+ \;\frac{\zeta_{n,\ell}}{1-\alpha}\frac{\zeta_{n,k}}{2-\alpha}\left((X_{t^{(n)}_{k+1}} - X_{t^{(n)}_{\ell}})^{2-\alpha} - (X_{t^{(n)}_{k}} - X_{t^{(n)}_{\ell}})^{2-\alpha} \right).\label{eq:LmRes4}
\end{align}
Thus, the desired result follows from \eqref{eq:LmRes1}, \eqref{eq:LmRes2}, \eqref{eq:LmRes3} and \eqref{eq:LmRes4}.

\section{The Cauchy problem}\label{sec:Cauchy}

With the dispersive estimates from Theorem \ref{thm:Strichartz} at hand, we are in position to solve the Cauchy problem of Equation \eqref{eq:MainMild}. The arguments that we use are standard and are based on a fixed-point strategy (see \cite{kato1987nonlinear,tsutsumi1987l2,cazenave2003semilinear}). 

Let $\psi_0\in L^2(\mathbb{R}^d)$, $\sigma < 2/dH$, $X\in\mathcal{A}$, $r\in\mathbb{R}$ such that $(r,2\sigma+2)$ is $H$-admissible and $T>0$. We consider the mapping $\Gamma$ given by
\begin{equation*}
\Gamma(\psi)(t,x) = P_{0,t}\psi_0(x) - i \lambda \int_0^{t} P_{s,t}|\psi|^{2\sigma}\psi(s,x)ds,\quad\forall (t,x)\in[0,T]\times\mathbb{R}^d.
\end{equation*}
Our goal is to prove that the mapping $\Gamma$ is a contraction in a closed subspace of $L^{r}([0,T];L^{2\sigma+2}(\mathbb{R}^d))$ in order to apply Banach's fixed-point theorem. The existence and uniqueness of a fixed point in $L^{r}([0,T];L^{2\sigma+2}(\mathbb{R}^d))$ will then solve the Cauchy problem of Equation \eqref{eq:MainMild}. The next proposition provides the necessary results to apply Banach's fixed-point theorem.

\begin{prop}
Denote $\mathcal{B}_{R,T}$ the closed ball of radius $R>0$ in $L^{r}([0,T];L^{2\sigma+2}(\mathbb{R}^d))$. There exists $T>0$ and $R>0$ such that 
\begin{enumerate}
\item $\Gamma$ is a contraction on $\mathcal{B}_{R,T}$,
\item $\Gamma(\mathcal{B}_{R,T})\subset \mathcal{B}_{R,T}$.
\end{enumerate}
\end{prop}
\begin{proof}

\noindent\textbf{First point:} We have, by using Theorem \ref{thm:Strichartz} and Hölder's inequality, $\forall \psi_1,\psi_2\in\mathcal{B}$,
\begin{align*}
\|\Gamma(\psi_1) - \Gamma(\psi_2)\|_{L^{r}([0,T];L^{2\sigma+2}(\mathbb{R}^d))} &\leq C_2|\lambda| \| |\psi_1|^{2\sigma}\psi_1 - |\psi_2|^{2\sigma}\psi_2\|_{L^{r'}([0,T];L^{\frac{2\sigma+2}{2\sigma+1}}(\mathbb{R}^d))}
\\ &\leq C_2|\lambda|(\| \psi_1\|_{L^{r'2\sigma}([0,T];L^{2\sigma+2}(\mathbb{R}^d))}^{2\sigma} + \| \psi_2\|_{L^{r'2\sigma}([0,T];L^{2\sigma+2}(\mathbb{R}^d))}^{2\sigma})
\\ & \hspace{13em}\times\|\psi_1-\psi_2\|_{L^{r'}([0,T];L^{2\sigma+2}(\mathbb{R}^d))}
\\ &\leq 2 C_2|\lambda| \|1\|_{L^{\gamma_1}([0,T])} \|1\|_{L^{\gamma_2}([0,T])}^{2\sigma}R^{2\sigma}\|\psi_1-\psi_2\|_{L^{r}([0,T];L^{2\sigma+2}(\mathbb{R}^d))}
\\ &\leq 2 C_2|\lambda| T^{1-(2\sigma+2)/r}R^{2\sigma}\|\psi_1-\psi_2\|_{L^{r}([0,T];L^{2\sigma+2}(\mathbb{R}^d))},
\end{align*}
where $\gamma_1,\gamma_2\in\mathbb{R}^+$ are such that
\begin{equation*}
\frac1{\gamma_1} + \frac{2\sigma}{\gamma_2} = 1 - \frac{2\sigma+2}{r}.
\end{equation*}
Since $(r,2\sigma+2)$ is $H$-admissible and $\sigma < 2/dH$, we deduce that
\begin{equation*}
\frac{2\sigma+2}{r} = \frac{dH}2 \sigma <1,
\end{equation*}
and, hence, $1-(2\sigma+2)/r>0$. Thus, by setting $T>0$ small enough to ensure that
\begin{equation}\label{eq:setT}
2 C_2|\lambda| T^{1-(2\sigma+2)/r}R^{2\sigma}< 1,
\end{equation}
this leads to the fact that $\Gamma$ is a contractive mapping.

\noindent\textbf{Second point:} We obtain, thanks to Theorem \ref{thm:Strichartz} and Hölder's inequality, $\forall \psi\in\mathcal{B}$,
\begin{align*}
\|\Gamma(\psi)\|_{L^{r}([0,T];L^{2\sigma+2}(\mathbb{R}^d))} &\leq C_1\|\psi_0\|_{L^2(\mathbb{R}^d)} + C_2|\lambda| \| |\psi|^{2\sigma+1} \|_{L^{r'}([0,T];L^{\frac{2\sigma+2}{2\sigma+1}}(\mathbb{R}^d))}
\\ &\leq C_1\|\psi_0\|_{L^2(\mathbb{R}^d)}  +  C_2|\lambda| \|\psi \|_{L^{r'(2\sigma+1)}([0,T];L^{2\sigma+2}(\mathbb{R}^d))}^{2\sigma+1}
\\&\leq C_1\|\psi_0\|_{L^2(\mathbb{R}^d)} +  C_2|\lambda|  \|1\|_{L^{\gamma_3}([0,T])}^{2\sigma+1}\|\psi \|_{L^{r}([0,T];L^{2\sigma+2}(\mathbb{R}^d))}^{2\sigma+1}
\\ &\leq C_1\|\psi_0\|_{L^2(\mathbb{R}^d)} +  C_2|\lambda|  T^{1 - \frac{2\sigma+2}r}R^{2\sigma+1},
\end{align*}
where $\gamma_3\in\mathbb{R}^+$ is such that
\begin{equation*}
\frac{2\sigma+1}{\gamma_3} = 1 - \frac{2\sigma+2}r.
\end{equation*}
Inequality \eqref{eq:setT} then leads to
\begin{equation*}
\|\Gamma(\psi)\|_{L^{r}([0,T];L^{2\sigma+2}(\mathbb{R}^d))} \leq C_1\|\psi_0\|_{L^2(\mathbb{R}^d)} + \frac12 R,
\end{equation*}
and, thus, by choosing
\begin{equation*}
R = 2C_1\|\psi_0\|_{L^2(\mathbb{R}^d)},
\end{equation*}
we obtain that $\mathcal{B}$ is stable by $\Gamma$.
\end{proof}
\bibliographystyle{plain}
\bibliography{Bibliography.bib}

\end{document}